\documentclass[12pt]{article}
\usepackage{amsfonts,amsmath,amssymb,mathtools,color,comment}
\usepackage[margin=1in]{geometry}
\usepackage[sorting=nyt, backend=bibtex]{biblatex}
\newcommand{\df}{\stackrel{\rm def}{=}}
\newcommand{\of}[1]{\left( #1 \right)}
\newcommand{\absvalue}[1]{\left| #1 \right|}
\newcommand{\set}[2]{\left\{\hspace{0.2ex} #1 \left|\: #2
\right. \right\}}
\newcommand{\vecset}[2]{\left(\hspace{0.2ex} #1 \left|\: #2 \right. \right)}
\newcommand{\injection}[2]{\colon #1 \rightarrowtail #2}
\newtheorem{theorem}{Theorem}
\newtheorem{definition}[theorem]{Definition}
\newtheorem{lemma}[theorem]{Lemma}

\newtheorem{claim}[theorem]{Claim}

\newtheorem{problem}{Problem}

\newcommand{\beginproof}{\medskip{\bf Proof.~}}
\newcommand{\finishproof}{\hspace{0.2ex}\rule{1ex}{1ex}}
\newenvironment{proof}{\beginproof}
{\unskip\nolinebreak\finishproof\par\medskip}

\title{On the Range of the Permanent of $(\pm1)$-Matrices}
\author{DeVon Ingram\thanks{University of Chicago, Department of Mathematics, {\tt dingram@math.uchicago.edu}} \and
Alexander Razborov\thanks{University of Chicago, {\tt razborov@uchicago.edu} and Steklov Mathematical Institute, {\tt razborov@mi-ras.ru}}}
\date{\today}

\begin{document}
\maketitle
\begin{abstract}
We establish a superpolynomial lower bound on the range of the permanent function on the set of $n\times n$ matrices with $\pm1$ entries.
\end{abstract}
\section{Introduction}
The \textbf{determinant} of an $n\times n$ matrix $M$ is given by the following formula:
\[
\det(M)\coloneqq\sum_{\sigma\in S_n}\text{sign}(\sigma)\prod_{i=1}^na_{i\sigma(i)}.
\]
The \textbf{permanent} is defined similarly, but without the factor of
sign$(\sigma)$:
\[
\mathrm{per}(M)\coloneqq\sum_{\sigma\in S_n}\prod_{i=1}^na_{i\sigma(i)}
\]
This seemingly innocuous modification leads to stark differences between
determinants and permanents, see e.g. \cite{minc1984permanents}.

A considerable amount of attention in the theory of
permanents has been paid to 0-1 matrices and ($\pm 1$)-matrices due to their
importance in combinatorics, numerical analysis and many other areas. The theory of permanents
of 0-1 matrices is rich and well developed, see e.g. \cite{GUTERMAN2018256} and the references cited therein. The
investigation of the permanents of ($\pm 1$)-matrices, which is the subject
of the current note, also has a long history starting with \cite{marcus1962inequalities,reich1971another,perfect1973positive} and continued in \cite{wang1974permanents,krauter1983permanenten,seifter1984upper,krauter1993permanents,wanless2005permanents,budrevich2015divisibility,GUTERMAN2018256,budrevich2018krauterconjecturepermanentstrue}.

Much of this work, both for 0-1 and ($\pm 1$)-matrices has concentrated on
studying the \textbf{range} of the permanent function, i.e the set of values
it can attain for a given $n$. As an example, the main result of \cite{wang1974permanents} asserts certain divisibility conditions by powers of 2 that
$\mathrm{per}(M)$ must satisfy for a ($\pm 1$)-matrix $M$ and \cite{budrevich2018divisibility} complemented this with non-divisibility conditions.
Krauter's conjecture concerning the minimum {\em positive} value of
$\mathrm{per}(M)$ is still open, although it was confirmed for smaller values
of $n$ in \cite{wanless2005permanents}.

One of the most basic questions one can ask in this direction is about the {\em cardinality} of the range or, in other words, ``how many values are attained by the permanent on the set of $(\pm 1)$-matrices?" Surprisingly, the only known lower bound on the cardinality of the aforementioned set seems
to be $n+1$ \cite{krauter1993permanents}. This is in sharp contrast with the situation for 0-1 matrices where the exponential lower bound $n!$ had been known long ago \cite{brualdirysercmt}. Perhaps, one explanation of this remarkable difference is that while the permanent of a 0-1 matrix (or, for that matter, any matrix with non-negative entries) is monotone in its entries, for $(\pm 1)$-matrices it is not true and local changes can bring about unpredictable results.

A better understanding of the range of the permanent function may yield insight on its study in various applications, such as combinatorial random matrix theory \cite{tao2008permanentrandombernoullimatrices,vu2020recentprogresscombinatorialrandom,kwan2022permanent} and quantum computing \cite{aaronson2010computationalcomplexitylinearoptics}. Also, a description of the overall structure of the range may lead to progress towards answering questions concerning the existence of large generalized arithmetic progressions contained within the range, which in turn would yield results on the (anti)concentration of the permanent \cite{nguyen2011optimalinverselittlewoodoffordtheorems}.

\medskip
The main result of this note (Theorem \ref{thm:main2})
drastically improves Krauter's lower bound to {\em super-polynomial in $n$}.
The proof uses some concepts and simple ideas from additive combinatorics and
Diophantine geometry.

\section{Notation and the main results}
\begin{itemize}
    \item $\mathbb{R}^{k\times n}$ - the set of all real $k\times n$
        matrices
    \item $\Omega_{k,n}\subseteq\mathbb{R}^{k\times n}$ - the set of all
        matrices with $\pm1$ entries
    \item $J_{k,n}\in\Omega_{k,n}$ - the matrix in which all entries are
        equal to 1
    \item $A\ast B\in\mathbb{R}^{(k+\ell)\times n}$ - the matrix obtained
        by row concatenation of $A\in\mathbb{R}^{k\times n}$ and
        $B\in\mathbb{R}^{\ell\times n}$
    \item For $\vec n =(n_1,\ldots,n_k)$ with $n_1+\ldots+n_k\leq n$, we
        let $\displaystyle B_{\vec{n}}\in\Omega_{k,n}$ be the matrix in
        which every column has at most one entry that is equal to -1 and
        the $i$th row has $n_i$ such entries (and hence there are $n_0\df
        n-(n_1+\ldots+n_k)$ columns with all ones)
    \item $[n]\df\{1,\dots,n\}$ - the first $n$ positive integers
    \item $S_{k,n}$ (for $k\leq n$) - the set of all injective functions
        $\sigma\injection{[k]}{[n]}$
    \item $(n)_k\df|S_{k,n}|=n(n-1)\cdots(n-k+1)$ - the $k$th falling
        factorial of $n$
    \item $\mathrm{per}(A)\df\displaystyle\sum_{\sigma\in
        S_{k,n}}A_{1\sigma(1)}\cdots A_{k\sigma(k)}$ - the permanent of
        $A\in\mathbb{R}^{k,n}$
    \item $r_{k,n}\df|\set{\mathrm{per}(A)}{A\in \Omega_{k,n}}|$ - the
        number of distinct values attained by the permanent of matrices in
        $\Omega_{k,n}$
    \item
        $e_\ell(x_1,\dots,x_k)\df\displaystyle\sum_{S\in\binom{[k]}{\ell}}\prod_{i\in
        S}x_i$ - the $\ell$th elementary symmetric polynomial in $k$
        variables
\end{itemize}
We will use lowercase letters to denote $1\times n$ matrices (e.g.
$j_n=(1,\dots,1)$ is the $n$-dimensional row vector with all entries equal to
1). We also write $\Omega_n\df\Omega_{n,n},S_n\df S_{n,n},$ and $r_n\df
r_{n,n}$.

\bigskip

Our main result is the following:

\begin{theorem}\label{thm:main}
For every $k\geq 1$ there exists a constant $\epsilon_k>0$ such that for all
$n\geq k+1$ we have
$$
r_{k+1,n} \geq \epsilon_k n^{k-2}.
$$
\end{theorem}

We remark that for $1\leq k\leq\ell\leq n$,
per$(A*J_{(\ell-k),n})=\mathrm{per}(A)\cdot(n-k)_{\ell-k}$. Therefore,
$r_{k,n}$ is increasing in $k\leq n$,\footnote{Remarkably, we have not been
able to prove that $r_{k,n}$ is increasing in $n$.} and in particular, $r_n\geq r_{k,n}$.

The precise asymptotics in Theorem \ref{thm:main} can be made explicit, 
and we will do so in Section \ref{sec:Asymptotics}. Along with the remark above 
this will imply 

\begin{theorem} \label{thm:main2}
$r_n\geq n^{\Omega\of{\frac{\log n}{\log\log n}}}$.
\end{theorem}

\section{Preliminaries}
We briefly review some standard definitions.

\subsection{Generalized arithmetic progressions}

This material can be found e.g. in \cite{tao2006additive}.

\begin{definition}[Generalized Arithmetic Progression]
A \textbf{generalized arithmetic progression (GAP)} of dimension $d$ is defined to be a set of the form
\[
\{x_0+\ell_1x_1+\cdots+\ell_dx_d:0\leq\ell_1<L_1,\dots,0\leq\ell_d<L_d\}
\]
where $x_0,x_1,\dots,x_d,L_1,\dots,L_d\in\mathbb{Z}$.
\end{definition}

Equivalently, a GAP is the projection of a $d$-dimensional ``box" in
$\mathbb{Z}^d$ to $\mathbb{Z}$. Given a GAP as defined above, we call the
product $L_1\dots L_d$ the \textbf{size} of the GAP. A GAP is \textbf{proper}
if its cardinality equals its size, i.e. the corresponding projection is
injective.

\subsection{Permanents} We now review several properties of permanents. Recall
that the \textbf{Laplace expansion} (or rather its straightforward analogue
for permanents) of an $n\times n$ matrix $B$ along the $i$th row is given by
\[
\mathrm{per}(B)=\sum_{j=1}^n b_{ij}p_{ij}
\]
where $b_{ij}$ is the $(i,j)$ entry of $B$ and $p_{ij}$ is the permanent of
the submatrix obtained by removing the $i$th row and $j$th column of $B$. We
note that this formula readily extends to rectangular $k\times n$ matrices,
where $k\leq n$.

We will also need the following explicit computation of the permanent of the
matrices $B_{n_1,\ldots,n_k}$ that will be the main building block in our
construction.

\begin{lemma} \label{lem:per_B}
For $n_1+\cdots+n_k\leq n$,
\[
\mathrm{per}(B_{n_1,\dots,n_k})=\sum_{\ell=0}^k\alpha_\ell(n)e_\ell(n_1,\dots,n_k),
\]
where
\[
\alpha_\ell(n)\df(-2)^\ell(n-\ell)_{k-\ell}
\]
\end{lemma}

\begin{proof}
Each summand in the permanent of a $k\times n$ matrix can be identified with
an injective function $\sigma:[k]\to[n]$ (given a row in $[k]$ as an input,
return a column in $[n]$ as an output). Set $I(\sigma)\df\set{i\in
[k]}{B_{i,\sigma(i)}=-1}$, where $B\df B_{n_1,\ldots,n_k}$. That is, let
$\Gamma_I\df \set{\sigma\in S_{k,n}}{I(\sigma)=I}$. Then
\begin{equation} \label{eq:per}
\mathrm{per}(B) = \sum_{I\subseteq [k]} (-1)^{|I|}\cdot |\Gamma_I|.
\end{equation}
Let $\Gamma_{\geq I} \df \bigcup_{J\supseteq I}\Gamma_J$. Then we can compute
$|\Gamma_I|$ via M\"obius inversion \cite{lovász2012large}
\begin{equation} \label{eq:mobius}
|\Gamma_I| =\sum_{J\supseteq I} (-1)^{|J\setminus I|}\cdot |\Gamma_{\geq J}|.
\end{equation}
One way to arrive at this equality is to note that every $\sigma\in\Gamma_I$
contributes 1 to both parts while every $\sigma\in \Gamma_J$ with $J\supset
I$ contributes zero. After plugging \eqref{eq:mobius} into \eqref{eq:per} and swapping the order of
summation, we see that
\[
\mathrm{per}(B) = \sum_{I\subseteq J\subseteq [k]} (-1)^{|J|}
\cdot |\Gamma_{\geq J}| = \sum_{J\subseteq [k]}(-2)^{|J|}\cdot |\Gamma_{\geq J}|.
\]
It only remains to note that $|\Gamma_{\geq J}| = \of{\prod_{i\in J}n_i}\cdot
(n-|J|)_{k-|J|}$ (values $\sigma(i)$ for $i\not\in J$ can be assigned
arbitrarily). Thus, splitting the sum according to $|J|$,
\[
\mathrm{per}(B) = \sum_{\ell=0}^k (-2)^\ell\cdot
(n-\ell)_{k-\ell}\cdot \sum_{J\in {[k]\choose\ell}}\prod_{i\in J}n_i =
 \sum_{\ell=0}^k \alpha_\ell(n)e_\ell(n_1,\ldots,n_k).
\]
\end{proof}

\section{Proof of Theorem \ref{thm:main}}
Let $k>0$. We fix a probability distribution $\mu=(\mu_1,\ldots,\mu_k)$ on
$[k]$, where $\mu_i$ are arbitrary positive distinct rationals summing to one.
For example, we can set
\begin{equation} \label{eq:good_mu_i}
\mu_i\df\frac i{{k+1\choose 2}}\ (1\leq i\leq k)
\end{equation}
albeit the exact choice of $\mu_i$ will become important only in Section
\ref{sec:Asymptotics}.

Let $d_k$ be the least common multiple of the denominators of
$\mu_1,\ldots,\mu_k$, let $n_0<d_k$ be the remainder of $n\bmod d_k$, and let
\begin{equation} \label{eq:mu_i}
n_i\df \mu_i(n-n_0).
\end{equation}
Then (provided $n\geq d_k$) $n_i$ are positive integers such that
$n_1+\ldots+n_k = n-n_0\leq n$. Hence we may consider the corresponding
matrix $B\df B_{n_1,\ldots,n_k}$, and we are going to show that
\[
|\set{\mathrm{per}(a\ast B)}{a\in \{\pm 1\}^n}| \geq \varepsilon_k n^{k-2},
\]
(where $\varepsilon_k$ will also be determined later), thus establishing
Theorem \ref{thm:main}.\\

Our first observation is that
\begin{equation} \label{eq:transformation}
|\set{\mathrm{per}(a\ast B)}{a\in \{\pm 1\}^n}| = |\set{\mathrm{per}(\varepsilon\ast B)}{\varepsilon\in \{0, 1\}^n}|.
\end{equation}
Indeed, for $a\in \{\pm 1\}^n$, let $\varepsilon_a= \frac 12(j_n-a);\
\varepsilon_a \in \{0,1\}^n$. Then, since the permanent function is linear in every individual row,
\[
\mathrm{per}(\varepsilon_a\ast B) = \frac 12(\mathrm{per}(j_n\ast B) - \mathrm{per}(a\ast B)).
\]
Hence the two sets featuring in \eqref{eq:transformation} are obtained from each other by an
invertible affine transformation and \eqref{eq:transformation} follows.

\medskip
Recall that $n_0=n-(n_1+\ldots+n_k)$. Assume without a loss of generality that the first $n_0$
columns in $B$ are all ones and let us restrict ourselves in
\eqref{eq:transformation} only to those $\epsilon$ for which
$\epsilon_1=\ldots=\epsilon_{n_0}=0$. Then the value
$\mathrm{per}(\varepsilon\ast B)$ can be computed by Laplace's
expansion in the first row. Namely, let $x_i\in [0,n_i]$ be the number of
those $j\in [n]$ for which $\varepsilon_j =1$ and $B_{ij}=-1$; thus,
$\varepsilon$ can be identified with the point $x\in [0,n_1]\times\ldots\in
[0,n_k]$. In addition, let
\[
p_i\df\mathrm{per}(B_{n_1,\ldots,n_{i-1},n_i-1,n_{i+1},\ldots,n_k}).
\]
Then the Laplace expansion amounts to the formula
\[
\mathrm{per}(\varepsilon\ast B) =\sum_{i=1}^k p_ix_i.
\]
In other words, the set of interest $\set{\mathrm{per}(\varepsilon\ast
B)}{\varepsilon\in \{0, 1\}^n}$ is simply the generalized $k$-dimensional
arithmetic progression
\[
\Gamma\df\set{\sum_{i=1}^k p_ix_i}{0\leq x_i\leq n_i}
\]
with basis $p_1,\ldots,p_k$. This set need not be proper, and the rest
of the proof amounts to showing that it contains a sufficiently large proper
sub-progression.

By Lemma \ref{lem:per_B},
\begin{eqnarray*}
p_i &=&\sum_{\ell=0}^k \alpha_\ell(n-1)e_\ell(n_1,\ldots, n_{i-1}, n_i-1, n_{i+1},\ldots, n_k)
\\ & = &
\sum_{\ell=0}^k \alpha_\ell(n-1)e_\ell(\mu_1 (n-n_0),\ldots, \mu_{i-1}(n-n_0), \mu_i(n-n_0)-1,
\mu_{i+1}(n-n_0),\ldots,
\mu_k(n-n_0));
\end{eqnarray*}
also note that $\alpha_\ell(n-1)$ is an integer polynomial in $n$ of degree
$k-\ell$. This observation allows us to view $p_i$ as a polynomial of
degree $\leq k$ in $\mathbb{Q}[n_0, n]$ (recall that the $\mu_i$ are fixed
throughout the argument).

The crucial part of the proof is the following result.
\begin{lemma}[main] \label{lem:main}
For any integer $n_0$, the resulting
system of polynomials $\bar p_1(n),\ldots, \bar p_k(n)\in \mathbb Q[n]$ has
co-rank $\leq 2$ over $\mathbb Q$. In other words, these polynomials span a
$\mathbb Q$-linear space of dimension $\geq k-2$.
\end{lemma}

We will prove Lemma \ref{lem:main} in the next section. For now, we will show how
it implies Theorem \ref{thm:main}. Let $N_k>0$ be an integer
such that $N_k\cdot p_1,\ldots, N_k\cdot p_k$ are integer
polynomials (say, $N_k=d_k^k$ will do), and let $M_k$ be the sum of absolute
values of the coefficients of $N_k\cdot \bar p_1(n),\ldots, N_k\cdot \bar
p_k(n)$, maximized over all choices of $n_0=0,1,\ldots, d_k-1$.  Let
\begin{equation}\label{eq:delta}
\delta_k\df \min\of{(4M_k)^{-1},\mu_1,\ldots,\mu_k}.
\end{equation}

Fix $n_0\in \{0,1,\dots,d_k-1\}$ and pick any $I\in {[k]
\choose k-2}$ such that $\set{\bar p_i(n)}{i\in I}$ are linearly independent. Consider the sub-cube $Q\subseteq \prod_{i=1}^k [0..n_i]$ defined as
follows: $(x_1,\ldots,x_k)\in Q$ iff $\forall i\in I(x_i\in [0..\delta_kn])$ and $\forall i\not \in I(x_i=0)$. Clearly, $|Q|\geq (\delta_kn)^{k-2}$. Then the
following claim implies Theorem 1, with
\begin{equation} \label{eq:epsilon}
\epsilon_k \df \delta_k^{k-2}.
\end{equation}

\begin{claim} \label{clm:final}
The restriction of $\Gamma$ onto $Q$ is proper.
\end{claim}

\begin{proof}
Let $x\neq y\in Q$; we have to prove that $\sum_{i=1}^k (y_i-x_i)p_i\neq 0$.
By our choice of $\vec\mu$ and $I$, the {\em integer} univariate
polynomial $\sum_{i=1}^k (y_i-x_i)N_k \bar p_i(n)$ is non-zero. Let $d$ be
its degree, thus
\[
\sum_{i=1}^k (y_i-x_i)N_k \bar p_i(n) = \beta_d n^d+\beta_{d-1}n^{d-1}+\ldots+
\beta_1n,
\]
where $\beta_1,\ldots,\beta_d\in \mathbb Z$ and $\beta_d\neq 0$. Moreover,
\[
|\beta_j|\leq (2\delta_k n) N_k \leq n/2,
\]
where the second inequality follows from the definition \eqref{eq:delta} of
$\delta_k$. Hence
\[
\absvalue{\sum_{i=1}^k (y_i-x_i)N_k \bar p_i(n)} \geq n^d \of{1-\frac 12-\frac 1{2n}-
\frac 1{2n^2}-\cdots}>0.
\]
\end{proof}

The proof of Theorem \ref{thm:main} is now complete, modulo lemma \ref{lem:main}.

\subsection{Proof of the main lemma}
Recall that the permanental minor $p_i$ (where $i\geq 1$) is given by
\[
\begin{split}
p_i&=\sum_{\ell=0}^k\alpha_\ell(n-1)e_\ell(n_1,\dots,n_{i-1},n_i-1,n_{i+1},\dots,n_k).
\end{split}
\]
We are actually going to prove that already the polynomials $\bar p_1(n)-\bar
p_i(n)\ (i>1)$, for any pairwise distinct $\mu_1,\ldots,\mu_k$ and any
integer $n_0$, span a subspace of dimension $\geq k-2$. To that end, we compute
\[
\begin{split} \label{eq:main_computation}
p_1-p_i&=\sum_{\ell=0}^k\alpha_\ell(n-1)\left[e_\ell(n_1-1,n_2,\dots,n_k)-e_\ell(n_1,\dots,n_{i-1},n_i-1,n_{i+1},\dots,n_k)\right]\\
&=\sum_{\ell=0}^k\alpha_\ell(n-1)\left[e_{\ell-1}(n_1,\dots,n_{i-1},n_{i+1},\dots,n_k)-e_{\ell-1}(n_2,\dots,n_k)\right]\\
&=(n_1-n_i)\sum_{\ell=2}^k\alpha_\ell(n-1)e_{\ell-2}(n_2,\dots,n_{i-1},n_{i+1},\dots,n_k)\\
&=(\mu_1-\mu_i)\sum_{\ell=2}^k\alpha_\ell(n-1)(n-n_0)^{\ell-1}e_{\ell-2}(\mu_2,\dots,\mu_{i-1},\mu_{i+1},\dots,\mu_k),
\end{split},
\]
where in the last line we used the definition \eqref{eq:mu_i} of $n_i$.

Let
\[
\widehat p_i(n) \df \frac{p_1-p_i}{\mu_1-\mu_i} \in\mathbb{Z}[n].
\]
Consider the $(k-1)\times (k-1)$ matrix $M$ with $M_{ij}\ (2\leq i\leq k,
0\leq j\leq k-2)$ being the coefficient of the $n^{k-j-1}$ in the $i$th
polynomial $\widehat p_i(n)$. We would like to show that this matrix has rank
$\geq k-2$ over $\mathbb{Q}$. Let us compute its entries.

First, note that $\frac{(n-\ell-1)!}{(n-k-1)!}(n-n_0)^{\ell-1}$ is a
polynomial of degree $k-1$ in $n$, so we can expand this quotient as an
alternating sum with coefficients given by elementary symmetric polynomials
in the values $\ell+1,\dots,k$ and $\ell-1$ copies of $n_0$:
\begin{eqnarray*}
\frac{(n-\ell-1)!}{(n-k-1)!}(n-n_0)^{\ell-1} &=&\underbrace{(n-n_0)(n-n_0)\ldots (n-n_0)}_{\ell-1\ \text{times}}
(n-\ell-1) \ldots (n-k)\\ &=&\sum_{j=0}^{k-1}(-1)^jn^{k-j-1}e_j(\underbrace{n_0,\ldots,n_0}_{\ell-1\ \text{times}}, \ell+1,\dots,k).
\end{eqnarray*}
Inserting this sum into the equation for $p_1-p_i$, we obtain
\[
\widehat p_i(n)=\sum_{\ell=2}^k\sum_{j=0}^{k-1}(-1)^j(-2)^\ell e_{\ell-2}(\mu_2,\dots,\mu_{i-1},\mu_{i+1},\dots,\mu_k)e_j(\underbrace{n_0,\ldots,n_0}_{\ell-1\ \text{times}},\ell+1,\dots,k)n^{k-j-1}.
\]
Switching the order of summation yields
\[
\widehat p_i(n) = \sum_{j=0}^{k-1}(-1)^jn^{k-j-1}\sum_{\ell=2}^k(-2)^\ell e_{\ell-2}(\mu_2,\dots,\mu_{i-1},\mu_{i+1},\dots,\mu_k)e_j(\underbrace{n_0,\ldots,n_0}_{\ell-1\ \text{times}},\ell+1,\dots,k).
\]
Thus
\[
M_{ij} = (-1)^j \sum_{\ell=2}^k(-2)^\ell e_{\ell-2}(\mu_2,\dots,\mu_{i-1},\mu_{i+1},\dots,\mu_k)e_j(\underbrace{n_0,\ldots,n_0}_{\ell-1\ \text{times}},\ell+1,\dots,k),
\]
where $2\leq i\leq k,\ 0\leq j\leq k-2$.\\

We now interpret this computation in a matrix form. For that, let $R_\ell\in\mathbb{Q}^{[2\dots k]}$ be given by
\[
(R_\ell)_i \df e_{\ell-2}(\mu_2,\dots,\mu_{i-1},\mu_{i+1},\dots,\mu_k),
\]
and let $S_\ell\in\mathbb{Q}^{[0..k-2]}$ be given by
\[
(S_\ell)_j \df (-1)^j\cdot e_j(\underbrace{n_0,\ldots,n_0}_{\ell-1\ \text{times}},\ell+1,\dots,k).
\]
Then
\[
M=\sum_{\ell=2}^k (-2)^\ell(R_\ell S_\ell^\top).
\]
The vectors $R_\ell$ are linearly independent over $\mathbb{Q}$ since the determinant of the matrix comprised of these vectors can be viewed as the Jacobian matrix corresponding to the transformation
\[
(\mu_2,\dots,\mu_k)\mapsto(e_{k-1}(\mu_2,\dots,\mu_k),\dots,e_1(\mu_2,\dots,\mu_k)),
\]
which is nonsingular at every point $(\mu_2,\dots,\mu_k)$ with pairwise distinct coordinates \cite{lascoux2002jacobians}. In fact, the determinant is proportional to
\[
\Delta\df\prod_{2\leq i<j\leq k}(\mu_i-\mu_j).
\]

Note that applying an invertible linear transformation on the left to the matrix
$M$ does not change its rank. Hence
\[
\mathrm{rk}(M)=\dim(\mathrm{Span}(S_\ell|2\leq\ell\leq k)).
\]
To compute the dimension of this space, we need the following generalization:
for $a\geq1,\ 2\leq \ell\leq k+1-a$, we let $S_{\ell,a}\in\mathbb{Q}^{[0\dots
k-a-1]}$ be given by
\[
(S_{\ell,a})_j\df(-1)^j\cdot e_j(\underbrace{n_0,\dots,n_0}_{\ell -1\text{ times}},\ell+a,\ell+a+1,\dots,k),
\]
so that $S_\ell=S_{\ell,1}$. Then we have the following recursion:
\begin{equation}
\begin{cases}
(S_{\ell+1,a} - S_{\ell,a})_0=0& (\text{since}\ (S_{\ell,a})_0=1\ \text{for all}\ a,\ell )\\
(S_{\ell+1,a} - S_{\ell,a})_j = (\ell+a-n_0) (S_{\ell, a+1})_{j-1},& j\geq 1.
\end{cases}
\end{equation}
Or, in the vector form,
\[
S_{\ell+1,a} - S_{\ell, a} = (\ell+a-n_0)(0\ast S_{\ell, a+1}).
\]

Now, since $S_{k+1-a,a}$ has 1 in the position $j=0$, we have the recursion
\begin{eqnarray*}
&& \dim(\text{Span}\vecset{S_\ell, a}{2\leq\ell\leq k+1-a})\\ &&\hspace{\parindent} =\dim(\text{Span}(S_{k+1-a,a},
\vecset{S_{\ell+1,a} - S_{\ell, a} }{2\leq\ell\leq k-a}))\\  && \hspace{\parindent} =
1+ \dim(\text{Span}\vecset{(\ell+a-n_0)\cdot S_{\ell, a+1} }{2\leq\ell\leq k-a}),
\end{eqnarray*}
and the only remaining problem is that the coefficient $\ell+a-n_0$ can be 0.
But this is actually not a problem since if $\ell+a=n_0$ then $S_{\ell,a+1} =
S_{\ell-1,a+1}$ and hence this vector still appears in the {\em set}
$\set{(\ell+a-n_0)\cdot S_{\ell, a+1} }{2\leq\ell\leq k-a}$ with coefficient
-1 and can be manually added to the right-hand side without changing its
dimension {\bf unless $\ell=2$}. By applying reverse induction on $a$, we prove that
\[
\dim(\text{Span}\vecset{S_{\ell, a}}{2\leq\ell\leq k+1-a}) \geq
\begin{cases}k-1-a& \text{if}\ a\leq n_0-2\\ k-a& \text{if}\ a\geq n_0-1\end{cases}
\]
for $a=1$, we obtain the required bound.\\

\section{Asymptotics} \label{sec:Asymptotics}
We first prove an estimate on  $\delta_k$.

Recall that
\[
\delta_k=\min\left((4M_k)^{-1},\mu_1,\dots,\mu_k\right),
\]
where $M_k$ is the sum of absolute
values of the coefficients of $d_k^k\cdot \bar p_1(n),\ldots, d_k^k\cdot \bar
p_k(n)$, maximized over all choices of $n_0=0,1,\ldots, d_k-1$. Tasked with estimating $\delta_k$, we will establish an upper bound on $M_k$.
\[
\begin{split}
M_k&\leq d_k^k\sum_{\ell=2}^k\left|(-1)^j(-2)^\ell e_{\ell-2}(\mu_2,\dots,\mu_{i-1},\mu_{i+1},\dots\mu_k)e_j(n_0,\dots,n_0,\ell+1,\dots,k)\right|\\&=d_k^k\sum_{\ell=2}^k2^\ell e_{\ell-2}(\mu_2,\dots,\mu_{i-1},\mu_{i+1},\dots\mu_k)e_j(n_0,\dots,n_0,\ell+1,\dots,k)\\
&\leq d_k^k\sum_{\ell=2}^k2^\ell d_k^{k-(\ell-2)}(d_k-1)\cdots(d_k-(\ell-2))\cdot\binom{k-2}{\ell-2}e_j(n_0,\dots,n_0,\ell+1,\dots,k)\\
&\leq d_k^k\sum_{\ell=2}^k2^\ell d_k^k\binom{k-2}{\ell-2}e_j(n_0,\dots,n_0,\ell+1,\dots,k)\\
&\leq d_k^k\sum_{\ell=2}^k2^\ell d_k^{k+j}\binom{k-2}{\ell-2}\binom{k-1}{j}\\
&\leq2^kd_k^{2k}\binom{k-1}{j}\sum_{\ell=2}^k\binom{k-2}{\ell-2}\\
&\leq(4d_k^2k)^k
\end{split}
\]

We now insist that $\mu_i$s are given by \eqref{eq:good_mu_i} so that
$d_k=\binom{k+1}{2}=\frac{k(k+1)}{2}\leq2k^2$. We conclude that
\[
M_k\leq(16k^5)^k.
\]
Consequently,
\[
\delta_k\geq\min\left(\frac{1}{4(16k^5)^k}\mu_1,\mu_2,\dots,\mu_k\right)=\frac{1}{4(16k^5)^k}.
\]

Thus $\delta_k\geq k^{-O(k)}$ and (cf. \eqref{eq:epsilon}) $\epsilon_k\geq k^{-O(k^2)}$.
Therefore the quantitative version of Theorem \ref{thm:main} can be stated as
$$
r_{k+1,n} \geq k^{-O(k^2)}\cdot n^{k-2}.
$$
Plugging in $k\df\lfloor \epsilon\frac{\log n}{\log\log n}\rfloor$ for a sufficiently
small $\epsilon>0$ makes the first term $\geq n^{-k/2}$ and proves Theorem \ref{thm:main2}.

\section{Open Questions}
This work presents a superpolynomial lower bound on the cardinality of the range of the permanent of $\pm1$ matrices. However, numerical evidence suggests that the true cardinality is much larger. It would thus be interesting to establish an exponential lower bound at the very least.

\begin{problem}
Show that $r_n$ grows exponentially in $n$.
\end{problem}

One can also consider the question of providing precise asymptotics on the range of the cardinality of the {\bf determinant} of $\pm1$ matrices. In this setting, there is an injective mapping between the set of determinants of \textit{normalized} (i.e. scaled by a factor of $2^{-(n-1)}$) $n\times n$ $\pm1$ matrices and the set of determinants of $(n-1)\times(n-1)$ 0-1 matrices under which the determinant is scaled by a factor of $2^{-(n-1)}$ and possibly a minus sign \cite{MO39826}. A lower bound of $\Omega(2^n/n)$ has been established \cite{SHAH2022229} in the 0-1 setting. Improving the lower bound for $\pm1$ matrices would imply a lower bound for 0-1 matrices, thus we have the following problem:

\begin{problem}
Establish a lower bound of $\Omega(2^n)$ on the cardinality of the range of the determinant of (normalized) $\pm1$ $n\times n$ matrices.
\end{problem}

One of the goals of this work was to improve our understanding of the structure of the range of the permanent. There are other seemingly innocuous structural questions that evaded all approaches thus far. The permanent of a $\pm1$ matrix is always divisible by a certain power of two. It is open whether this value is always attained:

\begin{problem}[Krauter]
Show that the minimum positive value attained by the permanent of a $\pm1$ matrix is equal to $2^{n-\lfloor\log_2(n)\rfloor-1}$ if $n=2^k-1$ for some $k\in\mathbb{N}$, and it is equal to $2^{n-\lfloor\log_2(n)\rfloor}$ otherwise.
\end{problem}

\noindent This problem is listed as conjecture 34 in \cite{Minc1982to1987}.\\

Numerical experiments also suggest that the family of matrices $B_{\vec{n}}$ considered in this work fails to produce every value attained by the permanent on $\pm1$ matrices in general. It would be interesting to find a family of matrices which can reproduce every value while being ``simpler''. A candidate family is the set of ``upper triangular" matrices (in the sense that any $-1$ entry appears on or above the diagonal).

\begin{problem}
Determine whether the range of the permanent on the subset of ``upper triangular" $\pm1$ $n\times n$ matrices coincides with the range of the permanent on $\Omega_{n,n}$.
\end{problem}

\subsection*{Acknowledgements}
We would like to thank Ilya Shkredov for answering our questions and providing useful
references.

\printbibliography

@book{lovász2012large,
  title={Large Networks and Graph Limits},
  author={Lov{\'a}sz, L{\'a}szlo},
  series={American Mathematical Society colloquium publications},
  year={2012},
  publisher={American Mathematical Society}
}

@article{seifter1984upper,
  title={Upper bounds for permanents of (1,- 1)-matrices},
  author={Seifter, Norbert},
  journal={Israel Journal of Mathematics},
  volume={48},
  pages={69--78},
  year={1984},
  publisher={Springer}
}

@book{minc1984permanents,
  title={Permanents},
  author={Minc, Henryk},
  volume={6},
  year={1984},
  publisher={Cambridge University Press}
}

@article{krauter1983permanenten,
  title={Permanenten - ein Kurzer Uberblick},
  author={Krauter, Arnold R},
  journal={S{\'e}minaire Lotharingien de Combinatoire},
  volume={9},
  pages={34},
  year={1983}
}

@article{GUTERMAN2018256,
title = {On the values of the permanent of (0,1)-matrices},
journal = {Linear Algebra and its Applications},
volume = {552},
pages = {256-276},
year = {2018},
author = {Guterman, Alexander E and Taranin, Konstantin A}
}

@article{perfect1973positive,
  title={Positive diagonals of $\pm$1-matrices},
  author={Perfect, Hazel},
  journal={Monatshefte f{\"u}r Mathematik},
  volume={77},
  pages={225--240},
  year={1973},
  publisher={Springer}
}

@article{wanless2005permanents,
  title={Permanents of matrices of signed ones},
  author={Wanless, Ian M},
  journal={Linear and Multilinear Algebra},
  volume={53},
  number={6},
  pages={427--433},
  year={2005},
  publisher={Taylor \& Francis}
}

@book{tao2006additive,
  title={Additive combinatorics},
  author={Tao, Terence and Vu, Van H},
  volume={105},
  year={2006},
  publisher={Cambridge University Press}
}

@article{nguyen2011optimalinverselittlewoodoffordtheorems,
      title={Optimal Inverse Littlewood-Offord theorems},
      author={Nguyen, Hoi and Vu, Van H},
      journal = {Advances in Mathematics},
      volume = {226},
      number = {6},
      pages = {5298-5319},
      year = {2011}
}

@inproceedings{aaronson2010computationalcomplexitylinearoptics,
      title={The Computational Complexity of Linear Optics},
      author={Aaronson, Scott and Arkhipov, Alex},
      booktitle={Proceedings of the forty-third annual ACM symposium on Theory of computing},
      pages={333--342},
      year={2011}
}

@article{budrevich2018krauterconjecturepermanentstrue,
      title={Kr\"auter conjecture on permanents is true},
      author={Budrevich, Mikhail V  and Guterman, Alexander E},
      journal = {Journal of Combinatorial Theory, Series A},
      volume = {162},
      pages = {306-343},
      year = {2019}
}

@article{budrevich2015divisibility,
  title={On divisibility for the permanents of ($\pm 1$)-matrices},
  author={Budrevich, Mikhail V and Guterman, Alexander E and Taranin, Konstantin A},
  journal={Zapiski Nauchnykh Seminarov POMI},
  volume={439},
  pages={26--37},
  year={2015},
  publisher={St. Petersburg Department of Steklov Institute of Mathematics, Russian~…}
}

@inproceedings{krauter1993permanents,
  title={Permanents of (1,-1)-matrices},
  author={Kr{\"a}uter, Arnold R},
  booktitle={Proceedings of the Topology and Geometry Research Center},
  volume={4},
  number={1},
  pages={151--204},
  year={1993},
}

@article{wang1974permanents,
  title={On permanents of (1,- 1)-matrices},
  author={Wang, Edward T},
  journal={Israel Journal of Mathematics},
  volume={18},
  pages={353--361},
  year={1974},
  publisher={Springer}
}

@book{brualdirysercmt,
  title={Combinatorial matrix theory},
  author={Brualdi, Richard A and Ryser, Herbert J},
  volume={39},
  year={1991},
  publisher={Springer}
}

@article{marcus1962inequalities,
  title={Inequalities for the permanent function},
  author={Marcus, Marvin and Newman, Morris},
  journal={Annals of Mathematics},
  volume={75},
  number={1},
  pages={47--62},
  year={1962},
  publisher={JSTOR}
}

@article{reich1971another,
  title={Another solution of an old problem of P{\'o}lya},
  author={Reich, Simeon},
  journal={The American Mathematical Monthly},
  volume={78},
  number={6},
  pages={649--650},
  year={1971},
  publisher={Taylor \& Francis}
}

@article{budrevich2018divisibility,
  title={On the Kr{\"a}uter-Seifter Theorem on Permanent Divisibility.},
  author={Budrevich, Michael V and Guterman, Alexander E and Taranin, Konstantin A},
  journal={Journal of Mathematical Sciences},
  volume={232},
  number={6},
  year={2018}
}

@article{tao2008permanentrandombernoullimatrices,
    title = {On the permanent of random Bernoulli matrices},
    journal = {Advances in Mathematics},
    volume = {220},
    number = {3},
    pages = {657-669},
    year = {2009},
    author = {Terence Tao and Van Vu},
}

@article{kwan2022permanent,
      title={On the permanent of a random symmetric matrix},
      author={Kwan, Matthew and Sauermann, Lisa},
      journal={Selecta Mathematica},
      volume={28},
      number={1},
      pages={15},
      year={2022},
      publisher={Springer}
}

@misc{vu2020recentprogresscombinatorialrandom,
      title={Recent progress in combinatorial random matrix theory},
      author={Vu, Van H},
      year={2020},
      eprint={2005.02797},
      archivePrefix={arXiv},
      primaryClass={math.CO},
      url={https://arxiv.org/abs/2005.02797},
}

@article{lascoux2002jacobians,
  title={Jacobians of symmetric polynomials},
  author={Lascoux, Alain and Pragacz, Piotr},
  journal={Annals of Combinatorics},
  volume={6},
  number={2},
  pages={169--172},
  year={2002},
  publisher={Birkh  user Verlag, Basel}
}

@article{SHAH2022229,
title = {Determinants of binary matrices achieve every integral value up to $\Omega(2^n/n)$},
journal = {Linear Algebra and its Applications},
volume = {645},
pages = {229-236},
year = {2022},
author = {Rikhav Shah},
}

@article{Minc1982to1987,
author = {Henryk Minc},
title = {Theory of permanents 1982–1985},
journal = {Linear and Multilinear Algebra},
volume = {21},
number = {2},
pages = {109--148},
year = {1987},
publisher = {Taylor \& Francis},
doi = {10.1080/03081088708817786},
URL = { 
        https://doi.org/10.1080/03081088708817786
},
eprint = { 
        https://doi.org/10.1080/03081088708817786
}
}

@MISC {MO39826,
    TITLE = {Maximum determinant of $\{0,1\}$-valued $n\times n$-matrices},
    AUTHOR = {Will Orrick (https://mathoverflow.net/users/484/will-orrick)},
    HOWPUBLISHED = {MathOverflow},
    NOTE = {URL:https://mathoverflow.net/q/39826 (version: 2019-01-16)},
    EPRINT = {https://mathoverflow.net/q/39826},
    URL = {https://mathoverflow.net/q/39826}
}
\end{document}